\documentclass[11pt]{article}

\usepackage[T1]{fontenc}
\usepackage[utf8]{inputenc}
\usepackage{lmodern}
\usepackage{amsmath,amssymb,amsthm}
\usepackage{microtype}
\usepackage[colorlinks=true,linkcolor=blue,citecolor=blue,urlcolor=blue,
pdftitle={Square-to-Cube Lindelöfness and Power Separations in Hattori Spaces},
pdfauthor={Xing-Yu Hu},
pdfsubject={Finite and countable powers, power separations, and one-sided obstructions in Hattori spaces},
pdfkeywords={Hattori space, Lindelöf space, Sorgenfrey line, finite powers, countable power, closed discrete subspace, affine subspace, monotone trace, discrete surface, Cantor space}]{hyperref}

\newtheorem{theorem}{Theorem}
\newtheorem{corollary}[theorem]{Corollary}
\newtheorem{lemma}[theorem]{Lemma}
\newtheorem{proposition}[theorem]{Proposition}
\newtheorem{question}[theorem]{Question}
\newtheorem{definition}[theorem]{Definition}

\newcommand{\R}{\mathbb R}
\numberwithin{equation}{section}

\title{Square-to-Cube Lindel\"ofness and Power Separations in Hattori Spaces}

\author{Xing-Yu Hu\thanks{Corresponding author.}\\
School of Mathematics and Statistics\\
Hanjiang Normal University\\
Shiyan, Hubei 442000, China\\
Email \texttt{huxingyu@hjnu.edu.cn}}

\date{}

\begin{document}

\maketitle

\begin{abstract}
Let \(H(A)\) be the Hattori space associated with \(A\subseteq\R\), and put
\(B=\R\setminus A\).  Kulesza announced at the 2018 Joint Mathematics
Meetings that Lindel\"ofness of \(H(A)^{2n}\) forces Lindel\"ofness of
\(H(A)^{2n+1}\) for every positive integer \(n\).  The meeting abstract
states this implication without a proof.  This paper proves the case \(n=1\)
in ZFC\@.
\[
H(A)^2\,\text{is Lindel\"of}\quad\Longrightarrow\quad
H(A)^3\,\text{is Lindel\"of}.
\]
The proof passes from an uncountable closed discrete subset of the cube to a
perfect set with injective coordinate projections and uses the one-sided
local bases to obtain three cocountable Pareto-maximality sets with countable
triple intersection.

The same one-sided geometry gives a pointwise half-space criterion for
affine discreteness and, when \(H(A)^2\) is Lindel\"of, a four-coordinate
obstruction for uncountable closed discrete monotone traces.  Known results on
mixed Sorgenfrey--Euclidean products lead to hereditary-Lindel\"of criteria
for finite powers, a countable-power characterization, and a finite-power
Lindel\"of reduction.
For every \(n\in\mathbb N\), the power \(H(A)^n\) is hereditarily
Lindel\"of if and only if the Sorgenfrey power \(B^n\) is Lindel\"of, and
\(H(A)^\omega\) is hereditarily Lindel\"of if and only if every finite
Sorgenfrey power of \(B\) is Lindel\"of.  A transfinite construction gives
\(A\subseteq\R\) for which \(H(A)^2\) and \(H(A)^3\) are Lindel\"of,
whereas \(H(A)^4\) contains a closed discrete affine line of cardinality
\(\mathfrak c\).  For this ZFC example, \(B^2\) is not Lindel\"of and the
second and third Hattori powers are not hereditarily Lindel\"of.
\end{abstract}


\medskip


\section{Introduction}

For \(A\subseteq\R\), put \(B=\R\setminus A\).  The Hattori space \(H(A)\)
is the real line with Euclidean local bases at points of \(A\) and Sorgenfrey
local bases at points of \(B\) \cite[Definition~5.1]{Hattori2010}.  Every
Hattori space is hereditarily Lindel\"of
\cite[Proposition~3.1(c)]{ChatyrkoHattori2013}.  Further properties were
studied by Kulesza \cite{Kulesza2017}, and Hattori later surveyed this family
of mixed topologies \cite{Hattori2026}.

Lin and Li posed a square product question relating Lindel\"ofness, normality,
and metrizability.  Under CH they gave an example with uncountable \(B\) and
Lindel\"of \(H(A)^2\) \cite[Question~4.1 and Example~4.2]{LinLi2025}.
Hern\'andez-Hern\'andez, Ram\'irez-Ch\'avez, and Rojas-Hern\'andez proved the
square characterization \cite[Theorem~6.7]{HHRR2025}.
\begin{equation}
\label{eq:square-characterization}
H(A)^2\,\text{is Lindel\"of}
\quad\Longleftrightarrow\quad
B\,\text{contains no subspace homeomorphic to }2^\omega.
\end{equation}
They asked for characterizations of the Hattori spaces whose cube, higher
finite powers, or countable power are Lindel\"of
\cite[Problem~7.8]{HHRR2025}.  At the 2018 Joint Mathematics Meetings,
Kulesza \cite{KuleszaJMM2018} announced
\begin{equation}
\label{eq:kulesza-parity}
H(A)^{2n}\,\text{is Lindel\"of}
\quad\Longrightarrow\quad
H(A)^{2n+1}\,\text{is Lindel\"of}
\qquad(n\ge1),
\end{equation}
together with odd-to-even separations and a finite-power versus
countable-power separation.  The meeting abstract states these conclusions
without proofs.

Theorem~\ref{thm:zfc-square-cube} proves the case \(n=1\) of
\eqref{eq:kulesza-parity} in ZFC\@.  Together with
\eqref{eq:square-characterization}, it gives
\begin{equation}
\label{eq:cube-characterization-intro}
\begin{aligned}
H(A)^3\,\text{is Lindel\"of}
&\quad\Longleftrightarrow\quad H(A)^2\,\text{is Lindel\"of}\\
&\quad\Longleftrightarrow\quad B\,\text{contains no subspace homeomorphic to }2^\omega.
\end{aligned}
\end{equation}

The proof uses structural properties of closed discrete subsets of a Hattori
cube.  Under the square hypothesis, every coordinate fiber is countable, so
an uncountable closed discrete set can be thinned until all three coordinate
projections are one-to-one.  Its Euclidean condensation set then contains a
perfect subset on which the three projections remain one-to-one.  Closed
discreteness and the one-sided local bases force three local
Pareto-maximality sets to be cocountable, while their triple intersection is
countable.

The one-sided geometry also constrains structured closed discrete subsets in
higher powers.  A pointwise half-space criterion characterizes affine
discreteness.  For nonempty \(A\), no positive-dimensional affine subspace is
discrete in the first three powers.  Under Lindel\"ofness of the square, an
uncountable closed discrete monotone trace requires at least two increasing
and two decreasing coordinates.  Both phenomena exhibit a four-coordinate
threshold.

Gruenhage and Nyikos proved the hereditary-Lindel\"of characterization used
for subspaces of finite products of Sorgenfrey and Euclidean lines
\cite[Theorem~6.9]{GruenhageNyikos1993}.  Hern\'andez-Hern\'andez,
Ram\'irez-Ch\'avez, and Rojas-Hern\'andez proved the needed product theorem
\cite[Proposition~6.6]{HHRR2025}.  These results give the reductions in
Section~\ref{sec:reductions}.  For every \(n\in\mathbb N\), the power
\(H(A)^n\) is hereditarily Lindel\"of if and only if \(B^n\) is Lindel\"of,
and \(H(A)^\omega\) is hereditarily Lindel\"of if and only if every finite
Sorgenfrey power of \(B\) is Lindel\"of.  A uniform finite-power reduction
also shows that, for \(n\ge3\), Lindel\"ofness of \(B^{n-1}\) implies
Lindel\"ofness of \(H(A)^n\).

Lin and Li's example was obtained under CH \cite[Example~4.2]{LinLi2025}.
Section~\ref{sec:zfc} gives a ZFC construction of \(A\subseteq\R\)
such that \(H(A)^2\) and \(H(A)^3\) are Lindel\"of, whereas \(H(A)^4\)
contains a closed discrete affine line of cardinality \(\mathfrak c\).  For this set \(A\),
\(B^2\) is not Lindel\"of, showing that non-Lindel\"ofness of the Sorgenfrey square
is necessary but not sufficient for non-Lindel\"ofness of the Hattori cube.
Neither \(H(A)^2\) nor \(H(A)^3\) is hereditarily Lindel\"of.

All finite and countable powers in what follows carry the product topology.

\section{Preliminaries}

Write \(2^\omega\) for the Cantor space and \(\mathfrak c\) for the
cardinality of the continuum.  The bounding number \(\mathfrak b\) is the
least cardinality of an unbounded family in \(\omega^\omega\) under eventual
domination.  Here \(\mathbb N=\{1,2,\ldots\}\).  In every finite-dimensional
Euclidean space, \(\|\cdot\|\) denotes the Euclidean norm and
\(B_{\mathrm E}(p,r)\) the open ball with center \(p\) and radius \(r>0\).
A point \(p\) is a Euclidean condensation point of a set \(S\) if every
Euclidean neighborhood of \(p\) meets \(S\) in uncountably many points.
Unless a different topology is explicitly specified, subsets of \(\R\) carry
the Euclidean subspace topology.

\begin{definition}[{\cite[Definition~5.1]{Hattori2010}}]
Let \(A\subseteq\R\), and put \(B=\R\setminus A\).  If \(x\in A\), the
usual Euclidean intervals
\[
(x-\varepsilon,x+\varepsilon),\qquad \varepsilon>0,
\]
form a local base at \(x\).  If \(x\in B\), the half-open intervals
\[
[x,x+\varepsilon),\qquad \varepsilon>0,
\]
form a local base at \(x\).  The resulting space is denoted by \(H(A)\).
\end{definition}

The Hattori topology refines the Euclidean topology, so \(H(A)\), all its
powers, and all their subspaces are \(T_1\).

\begin{theorem}[{\cite[Theorem~6.7]{HHRR2025}}]
\label{thm:square}
Let \(A\subseteq\R\) and \(B=\R\setminus A\).  Then \(H(A)^2\) is
Lindel\"of if and only if \(B\) contains no subspace homeomorphic to
\(2^\omega\).
\end{theorem}

\begin{lemma}[{\cite[Proposition~2.3]{ChatyrkoHattori2013}}]
\label{lem:second-countable}
Let \(A\subseteq\R\), and put \(B=\R\setminus A\).  Then \(H(A)\) is second
countable if and only if \(B\) is countable.
\end{lemma}

\begin{proposition}[{\cite[Lemma~3.2]{BurkeMoore1998}}]
\label{prop:discrete-surface}
Let \(X\subseteq\R\) carry the Sorgenfrey subspace topology and let
\(m\in\mathbb N\).  Then \(X^m\) is Lindel\"of if and only if it contains no
uncountable set \(D\) such that for all distinct \(x,y\in D\) there are
indices \(i,j\in\{1,\ldots,m\}\) with
\[
x_i<y_i
\qquad\text{and}\qquad
x_j>y_j.
\]
Such a set is called a discrete surface.
\end{proposition}

A space is subparacompact if every open cover has a \(\sigma\)-locally finite
closed refinement.  By Burke's characterization, this is equivalent to the
usual formulation in terms of \(\sigma\)-discrete closed refinements
\cite[Theorem~1.2]{Burke1969}.  A space is perfect if every closed subset is a
\(G_\delta\)-set, equivalently if every open subset is an \(F_\sigma\)-set.
A space is perfectly subparacompact if it is both perfect and
subparacompact \cite[Definition~2.2]{LinLi2025}.  This meaning of
``perfect'' is distinct from a perfect subset of a Euclidean space, which
means a closed subset with no isolated points.

\begin{proposition}[{\cite[Theorem~3.23]{LinLi2025}}]
\label{prop:perfect-subpara}
For every \(A\subseteq\R\), the countable power \(H(A)^\omega\) is perfectly
subparacompact.
\end{proposition}

\begin{lemma}[{\cite[Lemma~G]{GruenhageNyikos1993}}]
\label{lem:perfect-subpara-hereditary}
Every perfectly subparacompact space is hereditarily subparacompact.
\end{lemma}

\begin{lemma}
\label{lem:hereditary-subpara}
For every \(A\subseteq\R\) and every \(m\in\mathbb N\cup\{\omega\}\),
every subspace of \(H(A)^m\) is subparacompact.
\end{lemma}

\begin{proof}
Proposition~\ref{prop:perfect-subpara} and
Lemma~\ref{lem:perfect-subpara-hereditary} show that every subspace of
\(H(A)^\omega\) is subparacompact.  Every finite power is homeomorphic to a
subspace of \(H(A)^\omega\), which proves the finite cases.
\end{proof}

\begin{lemma}
\label{lem:subpara-extent}
A \(T_1\) subparacompact space is Lindel\"of if and only if every closed
discrete subspace is countable.
\end{lemma}

\begin{proof}
If the space is Lindel\"of, every closed discrete subspace is countable.
For the converse, the assumption that every closed discrete subspace is
countable means that the space is \(\omega_1\)-compact in the terminology of
Gruenhage and Nyikos.  They observe that every
subparacompact \(\omega_1\)-compact space is Lindel\"of
\cite[p.~584]{GruenhageNyikos1993}.
\end{proof}

Every nonempty Polish space without isolated points contains a subspace
homeomorphic to \(2^\omega\) \cite[Theorem~6.2]{Kechris1995}, and every
uncountable analytic subset of a Polish space contains a nonempty perfect
subset \cite[Theorem~29.1]{Kechris1995}.

\section{The square-to-cube implication}
\label{sec:square-cube}

The lemma below follows from Mycielski's perfect-set theorem for meager
relations \cite[Theorem~19.1]{Kechris1995}.  See also \cite{Mycielski1964}.

\begin{lemma}
\label{lem:simultaneous-perfect-injective}
Let \(m\in\mathbb N\), and let \(C\subseteq\R^m\) be a nonempty compact
perfect set.  Suppose that, for every \(i\in\{1,\ldots,m\}\) and
\(a\in\R\), the fiber
\[
\{x\in C:x_i=a\}
\]
has empty interior in \(C\).  Then there is a nonempty compact perfect set
\(Q\subseteq C\) such that every coordinate projection is one-to-one on
\(Q\).
\end{lemma}

\begin{proof}
For each \(i\in\{1,\ldots,m\}\), let
\[
R_i=\{(x,y)\in C^2:x_i=y_i\}.
\]
The relation \(R_i\) is closed in \(C^2\).  It has empty interior.  If
nonempty relatively open sets \(U,V\subseteq C\) satisfied
\(U\times V\subseteq R_i\), then, for any fixed \(y\in V\), the set
\(U\) would be contained in the coordinate fiber
\(\{x\in C:x_i=y_i\}\), contrary to the hypothesis.  Thus every
\(R_i\) is nowhere dense, and \(R=\bigcup_{i=1}^mR_i\) is meager in
\(C^2\).  The compact metric space \(C\) is Polish and has no isolated
points.  Mycielski's theorem gives a nonempty perfect set \(Q\subseteq C\).
Since \(Q\) is closed in the compact set \(C\), it is compact.  Put
\(\Delta_Q=\{(x,x):x\in Q\}\).  Then
\[
(Q^2\setminus\Delta_Q)\cap R=\emptyset.
\]
Hence distinct points of \(Q\) differ in every coordinate, so each
coordinate projection is one-to-one on \(Q\).
\end{proof}

\begin{theorem}
\label{thm:zfc-square-cube}
For every \(A\subseteq\R\),
\[
H(A)^2\,\text{is Lindel\"of}
\quad\Longrightarrow\quad
H(A)^3\,\text{is Lindel\"of}.
\]
\end{theorem}

\begin{proof}
Put \(B=\R\setminus A\), and suppose that the square is Lindel\"of while the
cube is not.  Lemmas~\ref{lem:hereditary-subpara} and~\ref{lem:subpara-extent}
give an uncountable closed discrete set
\(D\subseteq H(A)^3\).

For \(i\in\{1,2,3\}\), let \(\pi_i\) denote the \(i\)-th coordinate
projection.  For \(t\in\R\), the fiber
\[
D\cap\pi_i^{-1}(\{t\})
\]
is a closed discrete subspace of the closed slice
\(\pi_i^{-1}(\{t\})\cong H(A)^2\).  Every such fiber is countable.  Hence the
range of each coordinate projection on any uncountable subset of \(D\) is
uncountable.  Successive thinning of \(D\) along the three coordinate
projections, with one representative chosen from every nonempty fiber at
each stage, gives an uncountable set \(D^*\subseteq D\) whose three
coordinate projections are one-to-one.
Every subset of a closed discrete subspace is closed and discrete in the
ambient space, so \(D^*\) is still closed discrete.  Since \(\R^3\) is the
union of countably many bounded boxes, after a further thinning \(D^*\) may
be assumed bounded.

Let \(C\) be the set of Euclidean condensation points of \(D^*\).  Since
\(\R^3\) is second countable, the points of \(D^*\) that are not
condensation points of \(D^*\) form a countable set.  Each such point lies in
a member of a fixed countable base whose intersection with \(D^*\) is
countable.  Therefore
\[
D^*\setminus C
\]
is countable, so \(C\) is nonempty.  If \(p\notin C\), there is a
Euclidean open neighborhood \(U\ni p\) for which \(U\cap D^*\) is
countable.  Every point of \(U\) is then outside \(C\), so
\(\R^3\setminus C\) is open and \(C\) is closed.  Every point of
\(C\) lies in \(\overline{D^*}^{\R^3}\).  Since \(D^*\) is bounded,
so is \(C\).  Thus \(C\) is compact.  It has no isolated points.  If \(p\in C\) were isolated in \(C\),
a Euclidean neighborhood \(U\) with \(U\cap C=\{p\}\) would meet \(D^*\)
uncountably, whereas
\[
U\cap D^*\subseteq(D^*\setminus C)\cup\{p\}
\]
would be countable.  Thus \(C\) is a nonempty perfect subset of \(\R^3\).
The same argument shows that every nonempty relatively open subset of \(C\)
meets
\(D^*\cap C\) uncountably.

For \(i\in\{1,2,3\}\) and \(a\in\R\), the coordinate fiber
\[
F_i(a)=\{x\in C:x_i=a\}
\]
has empty interior in \(C\).  Otherwise a nonempty relatively open subset of
\(C\) contained in \(F_i(a)\) would meet \(D^*\cap C\) uncountably, contrary
to injectivity of \(\pi_i|_{D^*}\).  Lemma~\ref{lem:simultaneous-perfect-injective}
therefore gives a nonempty compact perfect set \(Q\subseteq C\) on which all
three coordinate projections are one-to-one.

For \(p\in Q\), put
\[
I_B(p)=\{i\in\{1,2,3\}:p_i\in B\}.
\]
Since \(D^*\) is closed discrete, choose a basic Hattori box \(U_p\) about
\(p\) such that \(U_p\cap D^*\subseteq\{p\}\).  Choose
\(\varepsilon_p>0\) so small that, whenever \(z\in\R^3\) satisfies
\(\|z-p\|<\varepsilon_p\), each coordinate \(z_i\) with
\(i\notin I_B(p)\) lies in the corresponding two-sided factor of \(U_p\),
and each coordinate \(z_i\) with \(i\in I_B(p)\) lies below the upper
endpoint of the corresponding one-sided factor of \(U_p\).  Since
\(U_p\cap(D^*\setminus\{p\})=\emptyset\), every
\(d\in D^*\setminus\{p\}\) with \(\|d-p\|<\varepsilon_p\) satisfies
\begin{equation}
\label{eq:strict-shadow-D}
d_i<p_i
\quad\text{for some }i\in I_B(p).
\end{equation}
Because \(p\) is a condensation point of \(D^*\), there is
\(d\in D^*\setminus\{p\}\) with \(\|d-p\|<\varepsilon_p\).
Equation~\eqref{eq:strict-shadow-D} then forces \(I_B(p)\) to be nonempty.

There is a Euclidean neighborhood \(W_p\) of \(p\) such that no
\(q\in Q\cap W_p\setminus\{p\}\) satisfies
\begin{equation}
\label{eq:strict-shadow-Q}
q_i>p_i
\quad\text{for every }i\in I_B(p).
\end{equation}
If no such neighborhood existed, choose \(q\in Q\setminus\{p\}\) close
enough to \(p\) that \(\|q-p\|<\varepsilon_p/2\) and
\eqref{eq:strict-shadow-Q} holds.  Let
\[
\eta=\min\{q_i-p_i:i\in I_B(p)\}>0.
\]
Since \(q\in C\), it is a condensation point of \(D^*\).  Hence there is
\(d\in D^*\setminus\{p\}\) with
\[
\|d-q\|<\min\{\eta/2,\varepsilon_p/2\}.
\]
Then \(d_i>p_i\) for every \(i\in I_B(p)\), while
\(\|d-p\|<\varepsilon_p\), contradicting
\eqref{eq:strict-shadow-D}.

For \(i=1,2,3\), put
\[
A_i=\{x\in Q:x_i\in A\}.
\]
Every \(A_i\) meets every nonempty perfect subset of \(Q\).  If a nonempty
perfect \(P\subseteq Q\) were disjoint from \(A_i\), then \(P\) would be
compact because it is closed in the compact set \(Q\).  The continuous
injection \(\pi_i|_P\) would therefore be a homeomorphism onto a nonempty
compact perfect subset of \(B\).  By \cite[Theorem~6.2]{Kechris1995}, such a
set contains a subspace homeomorphic to \(2^\omega\), contrary to
Theorem~\ref{thm:square}.

For \(i\in\{1,2,3\}\), let \(E_i\) be the set of points \(x\in Q\) for
which \(x_i\) has a right-hand gap in \(\pi_i[Q]\), meaning that
\(\pi_i[Q]\cap(x_i,x_i+\varepsilon)=\emptyset\) for some
\(\varepsilon>0\).  Each \(E_i\) is countable.  Apart from a
possible maximum, every right-gap point of the compact perfect set
\(\pi_i[Q]\) is the left endpoint of a component of its complement, and
there are only countably many such components.

For distinct \(i,j,k\in\{1,2,3\}\), let \(M_i\) consist of the points
\(x\in Q\) that are local Pareto maxima in coordinates \(j,k\).  Thus
\(x\in M_i\) if there is \(\delta>0\) such that
\[
\nexists\, y\in Q\setminus\{x\}\text{ such that }
\|y-x\|<\delta,\qquad y_j\ge x_j,\qquad y_k\ge x_k.
\]
The set \(M_i\) is Borel.  For \(n\in\mathbb N\), let \(M_{i,n}\) be
defined by the same condition with \(\|y-x\|\le1/n\).  For
\(r\in\mathbb N\), the set
\[
R_{n,r}=\{(x,y)\in Q^2:1/r\le\|y-x\|\le1/n,
            \ y_j\ge x_j,\ y_k\ge x_k\}
\]
is compact.  Write \(\operatorname{pr}_1:Q^2\to Q\) for projection onto
its first factor.  Then \(\operatorname{pr}_1(R_{n,r})\) is compact, and
\[
Q\setminus M_{i,n}
 =\bigcup_{r\in\mathbb N}\operatorname{pr}_1(R_{n,r}).
\]
Thus \(M_{i,n}\) is a \(G_\delta\) subset of \(Q\), and
\(M_i=\bigcup_{n\in\mathbb N}M_{i,n}\) is Borel.

The required inclusion is
\begin{equation}
\label{eq:AiMi}
A_i\subseteq M_i\cup E_j\cup E_k.
\end{equation}
Take \(x\in A_i\setminus(E_j\cup E_k)\).  If both \(x_j,x_k\in A\), then
\(I_B(x)=\emptyset\), contradicting \(I_B(x)\ne\emptyset\).  Suppose exactly
one of \(x_j,x_k\) lies in \(B\).  By symmetry in \(j\) and \(k\),
suppose \(x_j\in B\).  Then
\(I_B(x)=\{j\}\).  Hence \eqref{eq:strict-shadow-Q} excludes all
sufficiently near points of \(Q\) whose \(j\)-th coordinate is larger than
\(x_j\).  Since
\(x\notin E_j\) and \(\pi_j|_Q\) is a homeomorphism onto its image, points
of \(Q\) with larger \(j\)-th coordinate must approach \(x\), a
contradiction.  Hence \(x_j,x_k\in B\).  Equation
\eqref{eq:strict-shadow-Q} now excludes every sufficiently near \(y\ne x\)
with both \(y_j>x_j\) and \(y_k>x_k\).  Injectivity of \(\pi_j|_Q\) and
\(\pi_k|_Q\) excludes equality for \(y\ne x\), so \(x\in M_i\).  This
proves \eqref{eq:AiMi}.

Let \(P\subseteq Q\) be nonempty and perfect.  The Borel set
\(P\setminus(E_j\cup E_k)\) is uncountable and therefore contains a
nonempty perfect subset \cite[Theorem~29.1]{Kechris1995}.  That perfect
subset meets \(A_i\), so \eqref{eq:AiMi} gives \(M_i\cap P\ne\emptyset\).
Thus \(M_i\) meets every nonempty perfect subset of \(Q\).  Since \(M_i\)
is Borel, the perfect set theorem implies that \(Q\setminus M_i\) is
countable.  Each of \(M_1,M_2,M_3\) is therefore cocountable in \(Q\).

The intersection
\[
M_1\cap M_2\cap M_3
\]
is countable.  For each \(x\) in this intersection, choose \(n(x)\in\mathbb N\) such
that all three local Pareto conditions hold in the Euclidean ball of radius
\(1/n(x)\) about \(x\).  For \(n\in\mathbb N\), let
\[
S_n=\{x\in M_1\cap M_2\cap M_3:n(x)=n\}.
\]
If \(x,y\in S_n\) are distinct, then \(\|x-y\|\ge1/n\).  Otherwise their
three coordinate differences are nonzero, so one of \(x,y\) is larger than
the other in at least two coordinates.  The Pareto condition indexed by the
remaining coordinate is then violated at the point that is smaller in those
two coordinates.  Thus \(S_n\) is \(1/n\)-separated and hence finite in the
compact metric space \(Q\).
Therefore \(M_1\cap M_2\cap M_3=\bigcup_{n\in\mathbb N}S_n\) is
countable, contrary to the fact that the intersection of three cocountable
subsets of the uncountable perfect set \(Q\) is uncountable.
\end{proof}

The separation argument applies the pigeonhole principle to the signs of three
nonzero coordinate differences.

\begin{corollary}
\label{cor:cube-characterization}
For \(A\subseteq\R\) and \(B=\R\setminus A\), the following are equivalent.
\begin{enumerate}
\item \(H(A)^3\) is Lindel\"of.
\item \(H(A)^2\) is Lindel\"of.
\item \(B\) contains no subspace homeomorphic to \(2^\omega\).
\end{enumerate}
\end{corollary}

\begin{proof}
The implication (1)\(\Rightarrow\)(2) follows by coordinate projection.
The equivalence of (2) and (3) is Theorem~\ref{thm:square}, and
Theorem~\ref{thm:zfc-square-cube} gives (2)\(\Rightarrow\)(1).
\end{proof}

If \(B\) is analytic, then Corollary~\ref{cor:cube-characterization} and
the perfect set theorem \cite[Theorem~29.1]{Kechris1995} show that \(B\) is
countable exactly when \(H(A)^2\) is Lindel\"of, equivalently when
\(H(A)^3\) is Lindel\"of.  Under these equivalent conditions, \(H(A)\) is
second countable by Lemma~\ref{lem:second-countable}, and hence
\(H(A)^\omega\) is second countable.  Since every Borel subset of \(\R\) is
analytic, the conclusions also hold for Borel \(B\).

For \(A\subseteq\R\), if \(H(A)^2\) is Lindel\"of, then every discrete
subspace of \(H(A)^3\) that is analytic in the Euclidean topology is
countable, and no discrete subspace of the cube
contains a nonempty Euclidean perfect subset.  A nonempty Euclidean perfect
subset of \(\R^3\) contains a compact subspace homeomorphic to \(2^\omega\)
\cite[Theorem~6.2]{Kechris1995}.  That subspace is Euclidean closed and hence
closed in the finer Hattori product topology.  If it lay inside a discrete
subspace, it would be an uncountable closed discrete subspace of the
Lindel\"of cube, which is impossible.  The analytic assertion then follows
from the perfect set theorem \cite[Theorem~29.1]{Kechris1995}.

\section{Affine discreteness in finite powers}
\label{sec:affine}

\begin{lemma}
\label{lem:halfspaces}
Let \(k\in\mathbb N\), let \(q\in\{0,\ldots,k\}\), and let
\(r_1,\ldots,r_q\in\R^k\).  Then
\[
\{u\in\R^k:r_i\cdot u\ge0\text{ for }1\le i\le q\}
\]
contains a nonzero vector.
\end{lemma}

\begin{proof}
If \(q=0\), the asserted set is \(\R^k\), which contains nonzero vectors.
If the vectors \(r_1,\ldots,r_q\) do not span \(\R^k\), choose a nonzero
vector orthogonal to their span.  Suppose that \(q=k\) and that the vectors
are linearly independent.  Let \(R\) be the matrix whose rows are the
\(r_i\).  There is a vector \(u\) satisfying
\[
Ru=(1,\ldots,1)^{\mathsf T}.
\]
Then \(r_i\cdot u=1\) for every \(i\).
\end{proof}

\begin{theorem}
\label{thm:affine-criterion}
Let \(A\subseteq\R\), put \(B=\R\setminus A\), and let \(k,m\in\mathbb N\)
with \(k\le m\).  Let \(M\) be an \(m\times k\) real matrix of rank
\(k\), and let \(b\in\R^m\).  Define \(F:\R^k\to\R^m\) by
\[
F(t)=b+Mt\qquad(t\in\R^k)
\]
and write \(F_i\) for the \(i\)-th coordinate function of \(F\) and
\(r_i\in\R^k\) for the \(i\)-th row of \(M\).  Put
\[
L=F(\R^k)\subseteq H(A)^m
\]
and, for \(t\in\R^k\),
\[
I(t)=\{i\in\{1,\ldots,m\}:F_i(t)\in B\}.
\]
Call an index \(i\) active at \(t\) if \(i\in I(t)\).  For
\(t\in\R^k\), put
\[
C(t)=\bigcap_{i\in I(t)}\{u\in\R^k:r_i\cdot u\ge0\},
\]
where an empty intersection is understood to be \(\R^k\).  Then \(L\) is
closed in \(H(A)^m\), and \(F(t)\) is isolated in \(L\) if and only if
\begin{equation}
\label{eq:cone-criterion}
C(t)=\{0\}.
\end{equation}
Consequently, \(L\) is discrete if and only if~\eqref{eq:cone-criterion}
holds for every \(t\in\R^k\).
\end{theorem}

\begin{proof}
The set \(L\) is a Euclidean affine subspace of \(\R^m\), hence it is closed
in the Euclidean topology.  The topology of \(H(A)^m\) refines the Euclidean
product topology, so \(L\) is also closed in \(H(A)^m\).

Fix \(t\in\R^k\) and put \(x=F(t)\).  Assume first that
\eqref{eq:cone-criterion} holds.  For each \(i\in I(t)\), choose the
right-hand basic neighborhood \([x_i,x_i+1)\).  In the other coordinates take
the whole space \(H(A)\).  Their product is an open neighborhood \(U\) of
\(x\) in \(H(A)^m\).  If \(F(t+u)\in U\), then
\[
r_i\cdot u\ge0\qquad(i\in I(t)).
\]
Equation~\eqref{eq:cone-criterion} gives \(u=0\).  Thus \(U\cap L=\{x\}\).

Conversely, suppose that \(C(t)\) contains a nonzero vector \(u\).  Let
\(U=\prod_{i=1}^mU_i\) be a basic neighborhood of
\(x\).  If \(x_i\in A\), choose \(\varepsilon_i>0\) such that
\[
(x_i-\varepsilon_i,x_i+\varepsilon_i)\subseteq U_i.
\]
If \(x_i\in B\), choose \(\varepsilon_i>0\) such that
\[
[x_i,x_i+\varepsilon_i)\subseteq U_i.
\]
Since there are only finitely many coordinates, choose \(\lambda>0\) small
enough that
\[
|\lambda r_i\cdot u|<\varepsilon_i\quad(x_i\in A)
\]
and
\[
0\le\lambda r_i\cdot u<\varepsilon_i\quad(x_i\in B).
\]
Hence \(F(t+\lambda u)\in U\).  Since \(M\) has rank \(k\) and \(u\ne0\),
this point differs from \(F(t)\).  Therefore \(F(t)\) is not isolated in
\(L\).
\end{proof}

For affine lines the criterion has a simpler form.

\begin{corollary}
\label{cor:line-criterion}
Let \(A\subseteq\R\), put \(B=\R\setminus A\), let \(m\in\mathbb N\), and
let \(b,v\in\R^m\) with \(v\ne0\).  Put
\[
L=\{b+tv:t\in\R\}\subseteq H(A)^m.
\]
Then \(L\) is discrete if and only if, for every \(t\in\R\), there are
indices \(i,j\in\{1,\ldots,m\}\) such that
\[
v_i>0,\qquad b_i+t v_i\in B,
\]
and
\[
v_j<0,\qquad b_j+t v_j\in B.
\]
\end{corollary}

\begin{proof}
For \(k=1\), the intersection in~\eqref{eq:cone-criterion} is \(\{0\}\) if
and only if the active coefficients contain both a positive and a negative
number.
\end{proof}

\begin{theorem}
\label{thm:cube-affine-barrier}
Let \(A\subseteq\R\) be nonempty and let \(m\in\{1,2,3\}\).  Then
\(H(A)^m\) contains no discrete affine subspace of positive dimension.
\end{theorem}

\begin{proof}
Suppose that \(L=F(\R^k)\subseteq H(A)^m\) is a discrete affine subspace of
positive dimension, with notation as in Theorem~\ref{thm:affine-criterion}.
Zero rows impose no restriction in~\eqref{eq:cone-criterion}, so they may be
discarded when Lemma~\ref{lem:halfspaces} is applied.

If \(k=3\), then necessarily \(m=3\).  At each parameter \(t\),
Lemma~\ref{lem:halfspaces} shows that at least four nonzero active rows would
be needed for~\eqref{eq:cone-criterion}, which is impossible.

Suppose that \(k=2\).  The only case not already excluded by
Lemma~\ref{lem:halfspaces} is \(m=3\).  Then all three rows of
\(M\) must be nonzero and active at every parameter, since otherwise at most
two nonzero active rows would remain and Lemma~\ref{lem:halfspaces} would
contradict~\eqref{eq:cone-criterion}.  Fix one row \(r_i\).
The affine function
\[
t\longmapsto b_i+r_i\cdot t
\]
is onto \(\R\).  Choose \(a\in A\) and \(t\in\R^2\) with
\(b_i+r_i\cdot t=a\).  At this parameter the \(i\)-th row is not active, so
at most two nonzero active rows remain.  Lemma~\ref{lem:halfspaces} contradicts
\eqref{eq:cone-criterion}.

It remains to consider \(k=1\).  By Corollary~\ref{cor:line-criterion}, the
nonzero coefficients of the direction vector must include both signs.  This is
impossible when \(m=1\), so \(m\in\{2,3\}\).  One of the two sign classes
therefore consists of a single coordinate.  That coordinate must take values in \(B\) for every parameter.  The
corresponding affine coordinate function is onto \(\R\), so this forces
\(B=\R\), contrary to \(A\ne\emptyset\).
\end{proof}

For \(A=\emptyset\), the conclusion fails already in the square.  The
anti-diagonal \(\{(t,-t):t\in\R\}\) is a closed discrete affine line in the
Sorgenfrey plane.  Thus the hypothesis \(A\ne\emptyset\) in
Theorem~\ref{thm:cube-affine-barrier} cannot be omitted.

\section{Monotone traces}
\label{sec:monotone}

The affine result is local and linear.  Under Lindel\"ofness of the square,
a perfect-kernel argument gives an analogous four-coordinate obstruction for
monotone traces without assuming an interval parameter or continuity at the
outset.

\begin{theorem}
\label{thm:monotone-trace-barrier}
Let \(A\subseteq\R\), put \(B=\R\setminus A\), let \(m\in\mathbb N\), and
suppose that \(H(A)^2\) is Lindel\"of.  Let \(T\subseteq\R\) be uncountable
and let
\[
F=(f_1,\ldots,f_m):T\longrightarrow\R^m
\]
be one-to-one.  Suppose that, for each \(i\in\{1,\ldots,m\}\), the function
\(f_i\) is constant, strictly increasing, or strictly decreasing on \(T\).
If \(F[T]\) is closed and discrete in
\(H(A)^m\), then at least two coordinate functions are strictly increasing
and at least two are strictly decreasing.  In particular, \(m\ge4\).
\end{theorem}

\begin{proof}
Since the graph of \(F\) in \(\R^{m+1}\) is covered by countably many
bounded boxes, one of them meets the graph in an uncountable set.  Restrict
\(T\) accordingly.  Then \(T\) and all the sets \(f_i[T]\) are bounded.  The
restricted image \(F[T]\) is still closed and discrete because it is a subset
of the original closed discrete image.  Put
\(K=\overline T^{\R}\).  This set is compact because \(T\) is bounded.
Choose a compact interval \(J\supseteq K\).  For each strictly increasing
coordinate, define a
nondecreasing extension on \(J\) by
\[
\bar f_i(x)=\sup\bigl(\{f_i(t):t\in T,\ t\le x\}\cup\{\inf f_i[T]\}\bigr)
\qquad(x\in J).
\]
This function agrees with \(f_i\) on \(T\).  For a strictly decreasing
coordinate, apply the same construction to \(-f_i\) and then negate the
resulting extension.  Extend each constant coordinate as a constant function.
Denote the resulting monotone extensions by \(\bar f_i\).

Each monotone real-valued function has only countably many discontinuities.
Only countably many points of \(K\) fail to be approached by points of
\(T\) from both sides.  Apart from the possible endpoints of \(K\), any
point not approached from one side is an endpoint of a component of
\(\R\setminus K\).  Removing from \(K\) the discontinuity points of the
\(\bar f_i\), together with the points that fail this two-sided approach
property, leaves an uncountable Borel subset of \(\R\).  The perfect set theorem
\cite[Theorem~29.1]{Kechris1995} therefore gives a nonempty perfect set
\(Q\subseteq K\) such that every \(\bar f_i\) is continuous at every point of
\(Q\), and every point of \(Q\) is approached by points of \(T\) from both
sides.  Since \(Q\) is closed in the compact set \(K\), it is compact.

If \(f_i\) is strictly increasing on \(T\), then \(\bar f_i|_Q\) is strictly
increasing.  For \(q<r\) in \(Q\), choose \(s,t\in T\) with
\(q<s<t<r\).  Since \(\bar f_i\) is nondecreasing and agrees with \(f_i\)
on \(T\),
\[
\bar f_i(q)\le f_i(s)<f_i(t)\le\bar f_i(r).
\]
The decreasing case is analogous.  Thus
\[
\bar F=(\bar f_1,\ldots,\bar f_m)|_Q
\]
is continuous.  Since \(F\) is one-to-one, at least one \(f_i\) is
strictly monotone on \(T\).  By the preceding argument,
\(\bar f_i|_Q\) is strictly monotone.  Hence \(\bar F\) is one-to-one.

Suppose first that there is no strictly increasing coordinate.  Fix
\(q\in Q\) and choose a strictly increasing sequence \(t_n\in T\) with
\(t_n<q\) and \(t_n\to q\).  For each \(n\), choose \(s_n\in T\) with
\(t_n<s_n<q\).  Then, for every strictly decreasing coordinate \(i\),
\[
f_i(t_n)>f_i(s_n)\ge\bar f_i(q)
\quad\text{and}\quad
f_i(t_n)\longrightarrow\bar f_i(q).
\]
For each strictly decreasing coordinate \(i\), this convergence is from
above.  If \(\bar f_i(q)\in B\), the terms are eventually in every basic
interval \([\bar f_i(q),\bar f_i(q)+\varepsilon)\), and if
\(\bar f_i(q)\in A\), Euclidean convergence suffices.  Thus
\(f_i(t_n)\to\bar f_i(q)\) in \(H(A)\).  Constant coordinates also
converge.  Therefore
\[
F(t_n)\longrightarrow\bar F(q)
\]
in \(H(A)^m\).  Since \(F\) is one-to-one, the points \(F(t_n)\) are
pairwise distinct.  If \(\bar F(q)\notin F[T]\), this convergence contradicts
closedness of \(F[T]\).  If \(\bar F(q)\in F[T]\), it contradicts
discreteness.  Thus there is at least one strictly increasing coordinate.
The symmetric argument from the right, using a strictly decreasing sequence
\(u_n>q\) with \(u_n\in T\) and \(u_n\to q\), gives at least one strictly
decreasing coordinate.

Suppose that \(f_j\) is the unique strictly increasing coordinate.  If
\(\bar f_j(q)\in A\) for some \(q\in Q\), choose a strictly increasing
sequence \(t_n\in T\) with \(t_n<q\) and \(t_n\to q\).  Then
\(F(t_n)\to\bar F(q)\) in \(H(A)^m\), again a contradiction.  Hence
\[
\bar f_j[Q]\subseteq B.
\]
The continuous injection \(\bar f_j|_Q\) is a homeomorphism onto its
compact image.  Since \(Q\) has no isolated points, that image is a nonempty
perfect subset of \(B\).  By \cite[Theorem~6.2]{Kechris1995}, this contradicts
Theorem~\ref{thm:square}.  Thus there are at least two
strictly increasing coordinates.  The argument for strictly decreasing
coordinates is symmetric.
\end{proof}

The four-coordinate threshold in Theorem~\ref{thm:monotone-trace-barrier}
is attained by the affine line constructed in the proof of
Theorem~\ref{thm:zfc-separation}.  Two of its coordinate functions are
strictly increasing and two are strictly decreasing.

\section{Reductions to Sorgenfrey powers}
\label{sec:reductions}

For \(A\subseteq\R\) and \(B=\R\setminus A\), the subspace topology on
\(A\) inherited from \(H(A)\) is Euclidean, while the subspace topology on
\(B\) is the Sorgenfrey subspace topology
\cite[Lemma~2.1]{ChatyrkoHattori2013}.  Gruenhage and Nyikos proved that
for a subspace of a finite product whose factors are Sorgenfrey or Euclidean
lines,
\(\omega_1\)-compactness and hereditary Lindel\"ofness are equivalent
\cite[Theorem~6.9]{GruenhageNyikos1993}.  Hern\'andez-Hern\'andez,
Ram\'irez-Ch\'avez, and Rojas-Hern\'andez proved that the product of a
hereditarily Lindel\"of space with a cosmic space is hereditarily Lindel\"of
\cite[Proposition~6.6]{HHRR2025}.

\begin{corollary}
\label{cor:euclidean-sorgenfrey-transfer}
Let \(r,m\in\mathbb N\), let \(C\subseteq\R^r\) carry its Euclidean
subspace
topology, and let \(X\subseteq\R\) carry the Sorgenfrey subspace topology.
If \(X^m\) is Lindel\"of, then \(C\times X^m\) is hereditarily
Lindel\"of.  If \(C\ne\emptyset\) and \(C\times X^m\) is Lindel\"of,
then \(X^m\) is Lindel\"of.
\end{corollary}

\begin{proof}
Assume that \(X^m\) is Lindel\"of.  As a subspace of a finite power of the
Sorgenfrey line, it is \(\omega_1\)-compact and hence hereditarily Lindel\"of by
\cite[Theorem~6.9]{GruenhageNyikos1993}.  The Euclidean subspace \(C\) is
second countable and therefore cosmic.  The product theorem
\cite[Proposition~6.6]{HHRR2025} now shows that \(C\times X^m\) is
hereditarily Lindel\"of.  If \(C\ne\emptyset\) and \(C\times X^m\) is
Lindel\"of, then the continuous projection \(C\times X^m\to X^m\) shows
that \(X^m\) is Lindel\"of.
\end{proof}

\begin{corollary}
\label{cor:mixed-power-equivalence}
Let \(\emptyset\ne A\subseteq\R\), put \(B=\R\setminus A\), and let
\(k,m\in\mathbb N\).  Here \(A\) carries the Euclidean subspace topology,
and \(B\) carries the Sorgenfrey subspace topology.  The following conditions are equivalent.
\begin{enumerate}
\item \(B^m\) is Lindel\"of.
\item \(B^m\) is hereditarily Lindel\"of.
\item \(A^k\times B^m\) is Lindel\"of.
\item \(A^k\times B^m\) is hereditarily Lindel\"of.
\end{enumerate}
\end{corollary}

\begin{proof}
If \(B^m\) is Lindel\"of, then \(B^m\), viewed as a subspace of a finite
power of the Sorgenfrey line, is hereditarily Lindel\"of by
\cite[Theorem~6.9]{GruenhageNyikos1993}.  Applying
Corollary~\ref{cor:euclidean-sorgenfrey-transfer} with \(C=A^k\) and \(X=B\)
shows that \(A^k\times B^m\) is hereditarily Lindel\"of.  This proves
(1)\(\Rightarrow\)(2) and (1)\(\Rightarrow\)(4).  The implications
(4)\(\Rightarrow\)(3) and (2)\(\Rightarrow\)(1) follow from the definitions.
Since \(A\ne\emptyset\), projection onto the last \(m\) coordinates gives
(3)\(\Rightarrow\)(1).
\end{proof}

\begin{corollary}
\label{cor:finite-hereditary}
Let \(A\subseteq\R\), put \(B=\R\setminus A\), and let \(n\in\mathbb N\).
Then \(H(A)^n\) is hereditarily Lindel\"of if and only if \(B^n\), with its
Sorgenfrey product topology, is Lindel\"of.
\end{corollary}

\begin{proof}
If \(H(A)^n\) is hereditarily Lindel\"of, then its subspace \(B^n\) is
Lindel\"of.  Conversely, suppose that \(B^n\) is Lindel\"of.  If
\(A=\emptyset\), then \(H(A)^n=B^n\) is \(\omega_1\)-compact and hence
hereditarily Lindel\"of by \cite[Theorem~6.9]{GruenhageNyikos1993}.  Suppose
that \(A\ne\emptyset\).  For each
\(r\in\{1,\ldots,n\}\), the projection of \(B^n\) onto its first \(r\)
coordinates shows that \(B^r\) is Lindel\"of.  The stratum \(A^n\) is second
countable.  The Lindel\"of stratum \(B^n\) is \(\omega_1\)-compact and hence
hereditarily Lindel\"of by \cite[Theorem~6.9]{GruenhageNyikos1993}.  Every
mixed stratum is hereditarily Lindel\"of by
Corollary~\ref{cor:mixed-power-equivalence}.
These finitely many strata cover \(H(A)^n\).  Intersecting them with an
arbitrary subspace of \(H(A)^n\) expresses that subspace as a finite union of
Lindel\"of spaces.  Hence \(H(A)^n\) is hereditarily Lindel\"of.
\end{proof}

\begin{theorem}
\label{thm:finite-power-reduction}
Let \(A\subseteq\R\), put \(B=\R\setminus A\), and let \(n\in\mathbb N\)
with \(n\ge3\).  If \(B^{n-1}\), with its Sorgenfrey product topology, is
Lindel\"of, then \(H(A)^n\) is Lindel\"of.  Consequently,
\[
H(A)^n\text{ is not Lindel\"of}
\quad\Longrightarrow\quad
B^{n-1}\text{ is not Lindel\"of}.
\]
\end{theorem}

For \(n=3\), the conclusion also follows from
Corollaries~\ref{cor:finite-hereditary} and~\ref{cor:cube-characterization}.
The proof below treats all \(n\ge3\) uniformly.

\begin{proof}
The hypothesis implies that \(B^r\) is Lindel\"of for every
\(r\in\{1,\ldots,n-1\}\) by coordinate projection.  In particular, \(B^2\) is
Lindel\"of, so Corollary~\ref{cor:finite-hereditary} and
Theorem~\ref{thm:square} show that \(B\) contains no subspace homeomorphic to
\(2^\omega\).

If \(A=\emptyset\), then \(B=\R\) and the hypothesis would make the
Sorgenfrey plane Lindel\"of by projection, which is impossible.  Hence
\(A\ne\emptyset\).  Every stratum of \(H(A)^n\) with at least one coordinate
in \(A\) is, up to a coordinate permutation, either \(A^n\) or
\(A^{n-r}\times B^r\) for some \(r\in\{1,\ldots,n-1\}\).  The first is second
countable, and all the others are Lindel\"of by
Corollary~\ref{cor:mixed-power-equivalence}.  There are only finitely many
strata, so
\[
H(A)^n\setminus B^n
\]
is Lindel\"of.

Suppose that \(H(A)^n\) is not Lindel\"of.  Choose an open cover with no
countable subcover and refine it by basic open boxes.  Let \(\mathcal U\)
denote the resulting cover.  It also has no countable subcover.  A countable
subfamily \(\mathcal U_0\) covers
\(H(A)^n\setminus B^n\).  The remainder
\[
C=H(A)^n\setminus\bigcup\mathcal U_0
\]
is contained in \(B^n\) and is uncountable.  Otherwise, choosing one member
of \(\mathcal U\) for each point of the countable remainder would give a
countable subcover of \(H(A)^n\).  Every member of
\(\mathcal U_0\) is a finite product of Euclidean open intervals and
half-open intervals, so \(C\) is Borel in the Euclidean space \(\R^n\).
The perfect set theorem \cite[Theorem~29.1]{Kechris1995}, followed by
\cite[Theorem~6.2]{Kechris1995}, gives a compact subspace \(K\subseteq C\)
homeomorphic to \(2^\omega\).  At least one coordinate projection of \(K\)
has uncountable image.  This image is an uncountable compact subset of \(B\).  By
\cite[Theorem~29.1]{Kechris1995} and \cite[Theorem~6.2]{Kechris1995}, it
contains a subspace homeomorphic to \(2^\omega\), contradicting the fact that
\(B\) contains no subspace homeomorphic to \(2^\omega\).
\end{proof}

\begin{theorem}
\label{thm:countable-power-reduction}
Let \(A\subseteq\R\) and put \(B=\R\setminus A\).  The following conditions
are equivalent.
\begin{enumerate}
\item \(H(A)^\omega\) is hereditarily Lindel\"of.
\item \(H(A)^m\) is hereditarily Lindel\"of for every \(m\in\mathbb N\).
\item \(B^m\), with its Sorgenfrey product topology, is Lindel\"of for every
\(m\in\mathbb N\).
\end{enumerate}
\end{theorem}

\begin{proof}
Assume first that (3) holds.  Let
\(Y\subseteq H(A)^\omega\), and suppose that \(Y\) is not Lindel\"of.
By Lemmas~\ref{lem:hereditary-subpara} and~\ref{lem:subpara-extent}, \(Y\)
contains an uncountable closed discrete set \(D\).

For each \(x=(x_i)_{i<\omega}\in D\), choose a basic neighborhood isolating
\(x\) from \(D\).  Shrink this neighborhood if necessary, and choose a finite
set \(F_x\subseteq\omega\) and an integer \(n_x\ge1\) with the following
property.  For \(i<\omega\), define
\[
U_i(x,F_x,n_x)=
\begin{cases}
(x_i-1/n_x,x_i+1/n_x),&i\in F_x,\ x_i\in A,\\
[x_i,x_i+1/n_x),&i\in F_x,\ x_i\in B,\\
\R,&i\notin F_x.
\end{cases}
\]
Then the basic set
\[
Q(x,F_x,n_x)=\prod_{i<\omega}U_i(x,F_x,n_x)
\]
meets \(D\) only at \(x\).  There are only countably many pairs
\((F_x,n_x)\).  Hence there exist a finite set \(F\subseteq\omega\), an
integer \(n\ge1\), and an uncountable set \(D_0\subseteq D\) for which
\(F_x=F\) and \(n_x=n\) for every \(x\in D_0\).  Pass to a further
uncountable subset \(D_1\) such that the set
\[
I=\{i\in F:x_i\in B\}
\]
does not depend on \(x\in D_1\).

Let \(\pi_F:H(A)^\omega\to\R^F\) denote projection onto the coordinates
indexed by \(F\).  The map \(\pi_F\) is one-to-one on \(D_1\), since two
distinct points with the same coordinates on \(F\) would belong to each
other's isolating basic sets.  Thus \(\pi_F[D_1]\) is an uncountable subset
of the Euclidean space \(\R^F\).  Choose a Euclidean condensation point
\(p\) of this image.  Then choose an uncountable set \(E\subseteq D_1\) such
that
\[
\pi_F[E]\subseteq B_{\mathrm E}(p,1/(4n)).
\]
For distinct \(x,z\in E\), all coordinates indexed by \(F\) differ by less
than \(1/n\).  Since \(z\notin Q(x,F,n)\), some \(i\in I\) satisfies \(z_i<x_i\).
Interchanging \(x\) and \(z\) gives some \(j\in I\) with \(x_j<z_j\).
Necessarily \(i\ne j\), so \(|I|\ge2\).  Write
\(I=\{i_1<\cdots<i_s\}\), where \(s=|I|\ge2\), and define
\[
\pi_I(x)=(x_{i_1},\ldots,x_{i_s})\in\R^s.
\]
The same inequalities show that \(\pi_I\) is one-to-one on \(E\), and
\(\pi_I[E]\) is an uncountable
discrete surface in \(B^s\).  This contradicts
Proposition~\ref{prop:discrete-surface} and condition~(3).  Hence every
subspace \(Y\) of \(H(A)^\omega\) is Lindel\"of, proving (1).

Condition~(1) implies~(2).  Fix \(m\in\mathbb N\) and \(x_0\in\R\).  The
subspace of \(H(A)^\omega\) whose first \(m\) coordinates are arbitrary and
whose remaining coordinates equal \(x_0\) is homeomorphic to \(H(A)^m\), so
it is hereditarily Lindel\"of.  Condition~(2) implies~(3) because \(B^m\) is a
subspace of \(H(A)^m\) for every \(m\in\mathbb N\).
\end{proof}

\begin{corollary}
\label{cor:bounding-number}
Assume \(\mathfrak b=\omega_1\).  There is a set \(A\subseteq\R\) such
that \(B=\R\setminus A\) is uncountable, \(H(A)\) is not second countable,
and \(H(A)^\omega\) is hereditarily Lindel\"of.
\end{corollary}

\begin{proof}
Under the hypothesis \(\mathfrak b=\omega_1\), Bonanzinga, Giacopello, and
Szewczak construct an uncountable subspace \(B\) of the Sorgenfrey line all of
whose finite powers are Hurewicz and hence Lindel\"of
\cite[Example~6.1]{BonanzingaGiacopelloSzewczak2026}.
Put \(A=\R\setminus B\).  Theorem~\ref{thm:countable-power-reduction}
shows that \(H(A)^\omega\) is hereditarily Lindel\"of, while
Lemma~\ref{lem:second-countable} shows that \(H(A)\) is not second
countable.
\end{proof}

Under the continuum hypothesis, Michael had already constructed an
uncountable Sorgenfrey subspace \(B\) with \(B^\omega\) hereditarily
Lindel\"of \cite[Remark, p.~210]{Michael1971}, so every finite power
\(B^m\) is Lindel\"of.
For \(A=\R\setminus B\), Theorem~\ref{thm:countable-power-reduction}
and Lemma~\ref{lem:second-countable} give the conclusion of
Corollary~\ref{cor:bounding-number} directly.  The continuum hypothesis also
implies \(\mathfrak b=\omega_1\).

\section{A ZFC separation between the third and fourth powers}
\label{sec:zfc}

The following construction attains the four-coordinate threshold.

\begin{lemma}
\label{lem:shift-avoiding}
There is a set \(A\subseteq\R\) of cardinality \(\mathfrak c\) that meets
every nonempty perfect subset of \(\R\) and satisfies
\[
A\cap(A+n)=\emptyset
\qquad(n\in\mathbb Z\setminus\{0\}).
\]
\end{lemma}

\begin{proof}
The family of nonempty perfect subsets of \(\R\) has cardinality
\(\mathfrak c\).  There are at most \(\mathfrak c\) closed subsets of the
second-countable space \(\R\), while the translates of the Cantor set by
distinct real numbers are distinct nonempty perfect sets.  Enumerate the
nonempty perfect subsets as
\[
\{P_\alpha:\alpha<\mathfrak c\}.
\]
Every \(P_\alpha\) has cardinality \(\mathfrak c\).  Recursively choose
\[
a_\alpha\in P_\alpha\setminus
\bigcup_{\beta<\alpha}(a_\beta+\mathbb Z).
\]
At stage \(\alpha<\mathfrak c\), the excluded set has cardinality at most
\(\max\{\aleph_0,|\alpha|\}<\mathfrak c\), so the choice is possible.  Put
\[
A=\{a_\alpha:\alpha<\mathfrak c\}.
\]
The selected points are distinct, so \(|A|=\mathfrak c\), and \(A\) meets every
nonempty perfect subset of \(\R\).  If distinct points of \(A\) differed by
an integer, the one chosen later would have been excluded at its stage.
Thus the displayed shift-avoidance property holds.
\end{proof}

The construction is a transfinite integer-shift avoidance argument.  For
Sierpi\'nski's construction technique on subsets of \(\R\), see
\cite{vanMill1992}.

\begin{theorem}
\label{thm:zfc-separation}
There is a set \(A\subseteq\R\) such that, with \(B=\R\setminus A\),
\begin{enumerate}
\item \(|B|=\mathfrak c\) and \(B\) contains no subspace homeomorphic to
\(2^\omega\).
\item \(H(A)^2\) and \(H(A)^3\) are Lindel\"of.
\item \(B^2\), with its Sorgenfrey product topology, contains a discrete
surface of cardinality \(\mathfrak c\) and is not Lindel\"of.
\item \(H(A)^4\) contains a closed discrete affine line of cardinality
\(\mathfrak c\).  Hence \(H(A)^n\) is not Lindel\"of for every integer
\(n\ge4\), and \(H(A)^\omega\) is not Lindel\"of.
\end{enumerate}
\end{theorem}

\begin{proof}
Choose \(A\) as in Lemma~\ref{lem:shift-avoiding}.  Since \(A\) meets every
nonempty perfect subset of \(\R\), the complement \(B\) contains no subspace
homeomorphic to \(2^\omega\).  Also \(A+1\subseteq B\), so
\(|B|=\mathfrak c\).  Theorems~\ref{thm:square} and~\ref{thm:zfc-square-cube}
give Lindel\"ofness of \(H(A)^2\) and \(H(A)^3\).

For \(i,j\in\{0,1\}\), put
\[
T_{ij}=\{t\in\R:t+i\in B\text{ and }-t+j\in B\}.
\]
The four sets \(T_{ij}\) cover \(\R\).  By
Lemma~\ref{lem:shift-avoiding}, neither \(\{t,t+1\}\) nor
\(\{-t,-t+1\}\) can be contained in \(A\).  Thus at least one point of each
pair belongs to \(B\).  Hence some \(T_{ij}\) has
cardinality \(\mathfrak c\).  For that pair, the set
\[
\{(t+i,-t+j):t\in T_{ij}\}\subseteq B^2
\]
is a discrete surface of cardinality \(\mathfrak c\).  If \(s<t\), the
first coordinate increases and the second decreases.
Proposition~\ref{prop:discrete-surface} therefore shows that \(B^2\) is not
Lindel\"of.

Consider the affine line
\[
D=\{(t,t+1,-t,-t+1):t\in\R\}\subseteq\R^4.
\]
Its direction vector is \((1,1,-1,-1)\).  For every \(t\), at least one of
\(t,t+1\) belongs to \(B\), and at least one of \(-t,-t+1\) belongs to
\(B\).  Corollary~\ref{cor:line-criterion} shows that \(D\) is discrete in
\(H(A)^4\), while Theorem~\ref{thm:affine-criterion} shows that it is closed.
Thus \(D\) is a closed discrete subspace of cardinality \(\mathfrak c\), so
\(H(A)^4\) is not Lindel\"of.  Projection onto the first four coordinates
then gives the assertions for every \(n\ge4\) and for the countable power.
\end{proof}

Together with Theorem~\ref{thm:cube-affine-barrier},
Theorem~\ref{thm:zfc-separation} shows that \(4\) is the least exponent \(m\)
for which there exists a nonempty \(A\subseteq\R\) such that \(H(A)^m\)
contains a closed discrete affine line.

For the set \(A\) in Theorem~\ref{thm:zfc-separation}, the Sorgenfrey square
\(B^2\) is non-Lindel\"of while the Hattori cube \(H(A)^3\) is Lindel\"of.
Theorem~\ref{thm:finite-power-reduction} shows that non-Lindel\"ofness of
\(B^2\) is necessary for non-Lindel\"ofness of \(H(A)^3\), whereas this
construction shows that it is not sufficient.

The same example also shows that \(H(A)^2\) and \(H(A)^3\) are Lindel\"of
but not hereditarily Lindel\"of.  The subspace \(B^2\subseteq H(A)^2\) is
not Lindel\"of, and for any \(x_0\in\R\), the subspace
\(B^2\times\{x_0\}\subseteq H(A)^3\) is homeomorphic to \(B^2\).

\section{Further questions}

Corollary~\ref{cor:finite-hereditary} and
Theorem~\ref{thm:countable-power-reduction} characterize hereditary
Lindel\"ofness of the finite and countable Hattori powers.  The questions
below concern Lindel\"ofness itself.

Corollary~\ref{cor:cube-characterization} gives the cube characterization
requested in \cite[Problem~7.8]{HHRR2025}, while
Theorem~\ref{thm:zfc-separation} shows that the square obstruction no longer
controls powers \(n\ge4\).  Kulesza's meeting abstract
\cite{KuleszaJMM2018} announces the general even-to-odd implication
\eqref{eq:kulesza-parity}.  Theorem~\ref{thm:zfc-square-cube} proves its
first case.

\begin{question}
\label{q:power-characterization}
For each fixed integer \(n\ge4\), characterize the sets \(A\subseteq\R\)
for which \(H(A)^n\) is Lindel\"of.  Also characterize the sets \(A\) for
which \(H(A)^\omega\) is Lindel\"of.  In particular, with
\(B=\R\setminus A\), determine structural conditions on the Sorgenfrey
stratum \(B\) that account for the general even-to-odd implication announced
by Kulesza and for the possible odd-to-even losses of Lindel\"ofness.
\end{question}

\end{document}